\documentclass[12pt]{amsart}
\usepackage[all]{xy}
\usepackage{mathrsfs}
\usepackage{amssymb}
\usepackage{amsthm}
\usepackage{algorithmicx}
\usepackage{tikz}
\usepackage[english]{babel}
\usepackage{graphics}
\usepackage{graphicx}
\graphicspath{ {images/} }
\usepackage{epstopdf}
\usepackage{amsmath}
\usepackage{cite} 
\calclayout                   
\newtheorem{dummy}{anything}[section]
\newtheorem{theorem}[dummy]{Theorem}

\newtheorem*{theorem*}{Main Theorem}

\newtheorem{lemma}[dummy]{Lemma}

\newtheorem{proposition}[dummy]{Proposition}

\newtheorem{corollary}[dummy]{Corollary}

\theoremstyle{definition}
\newtheorem{definition}[dummy]{Definition}

\newtheorem{remark}[dummy]{Remark}

\def\:{\mkern 1.2mu \colon}
\newcommand{\mmatrix}[4]{\left (\vcenter
{\xymatrix@C-2pc@R-2pc{#1&#2\\#3&#4} } \right )}

\numberwithin{equation}{section} 

\begin{document}
\title[Generalized Chillingworth Classes On Subsurface Torelli Groups]
{Generalized Chillingworth Classes On Subsurface Torelli Groups}


\author{Hat\.{I}ce \"{U}nl\"{U} Ero\u{G}lu}

\address{Department of Mathematics-Computer Science
\newline\indent
Necmettin Erbakan University
\newline\indent
Konya, Turkey}  
\email{hueroglu{@}erbakan.edu.tr}

\date{\today}

\begin{abstract}
The contraction of the image of the Johnson homomorphism is called the Chillingworth class. In this paper, we derive a combinatorial description of the Chillingworth class for Putman's subsurface Torelli groups. We also prove the naturality and uniqueness properties of the map whose image is the dual of the Chillingworth classes of the subsurface Torelli groups. Moreover, we relate the Chillingworth class of the subsurface Torelli group to the partitioned Johnson homomorphism.
		
\end{abstract}

\maketitle

\section{INTRODUCTION} 
The Torelli group of an oriented surface with genus $g$ and $n$ boundary components $\Sigma_{g,n}$, $\mathcal{I}(\Sigma_{g,n})$, 
is the normal subgroup of the mapping class group $\mathcal{M}(\Sigma_{g,n})$ 
of $\Sigma_{g, n}$ that acts trivially on $H_1(\Sigma_{g,n}; \mathbb{Z})$. 
In the study of the Torelli group, the Johnson homomorphism has an important role. The Johnson homomorphism determines the 
abelianization of $\mathcal{I}(\Sigma_{g,1})$ mod torsion \cite{Johnson2}. As no finite presentation for the Torelli group is known, the 
finiteness information inherent in the abelianization of the Torelli 
group is a fundamental tool.

The tensor contraction of the image of the 
Johnson homomorphism is the Chillingworth class. The Chillingworth class was first defined by Earle \cite{Earle} by using complex analytic methods. Johnson in \cite{Johnson} defined the Chillingworth 
class by considering Chillingworth's conjecture in \cite{Chillingworth2}. In \cite{Johnson}, Johnson 
called the homomorphism $t:\mathcal{I}(\Sigma_{g,1})\rightarrow H_1(\Sigma_{g,1}; \mathbb{Z})$ 
sending each $f\in \mathcal{I}(\Sigma_{g,1})$ to the 
Chillingworth class of $f$ the Chillingworth homomorphism. 

In \cite{Putman1}, Putman defined the subsurface Torelli groups in order to use inductive arguments 
in the Torelli group.
An embedding of a subsurface $\Sigma_{g, n}$ into a larger surface 
$\Sigma_{g'}$ gives a partition $\mathcal{P}$ of the boundary components 
of $\Sigma_{g, n}$ and this partition records which of 
the boundary components of $\Sigma_{g, n}$ become homologous in $\Sigma_{g'}$ \cite{Church}. 
Putman \cite{Putman1} defined the subsurface Torelli group $\mathcal{I}(\Sigma_{g, n}, \mathcal{P})$ 
by restricting $\mathcal{I}(\Sigma_{g'})$ to $\Sigma_{g, n}$.  
The subsurface Torelli groups $\mathcal{I}(\Sigma_{g, n}, \mathcal{P})$ 
restore functoriality and are therefore of central 
importance to the study of the Torelli group.


In this paper, we construct a combinatorial description of the Chillingworth 
class of the subsurface Torelli groups via winding numbers in the projective 
tangent bundle of $\Sigma_{g,n}$. Given the definition of Putman's subsurface Torelli groups, 
the difficulty in finding a combinatorial description via winding numbers is to 
make sense of the winding number of an arc with end points on the boundary of the 
subsurface. By defining a difference cocycle on the projective tangent bundle of 
the surface we are able to make sense of the winding number of the difference of two arcs. 
 
The rest of this paper is structured as follows:

In Section $2$, we review basic definitions related to the Torelli group and the subsurface Torelli groups.

In Section $3$, we construct a well-defined map 
$\widetilde{e}_X: \mathcal{I}(\Sigma_{g,n}, \mathcal{P})\rightarrow H_1^{\mathcal{P}}(\Sigma_{g,n}, \mathbb{Z})$ using the projective tangent bundle of $\Sigma_{g,n}$. 
Here, $X$ is a nonvanishing vector field on $\Sigma_{g,n}$ and 
$H_1^{\mathcal{P}}(\Sigma_{g,n}, \mathbb{Z})$ denotes the homology group 
defined by Putman \cite{Putman1}. We show that $\widetilde{e}_X$ is a 
homomorphism. We define a symplectic basis for the homology group 
$H_1^{\mathcal{P}}(\Sigma_{g,n}, \mathbb{Z})$ and call the dual of 
$\widetilde{e}_X(f)$ 
\textit{the Chillingworth class of $f$}. One reason for calling this dual 
the Chillingworth class, is that it is shown to factor through the partitioned 
Johnson homomorphism. Therefore, we obtain a combinatorial description of 
the Chillingworth class of the subsurface Torelli groups using the projective 
tangent bundle of $\Sigma_{g,n}$.

We use the Torelli category $\mathcal{T}\mathrm{Surf}$ defined by Church \cite{Church}, which 
is the refinement of the category $\mathrm{TSur}$ defined by Putman \cite{Putman1}. 
The Torelli group is a functor from $\mathcal{T}\mathrm{Surf}$ to the category of groups and 
homomorphisms \cite{Putman1}. For a morphism 
$i:(\Sigma_{g,n}, \mathcal{P})\rightarrow (\Sigma_{g',n'}, \mathcal{P}')$ of 
$\mathcal{T}\mathrm{Surf}$ and a nonvanishing vector field $X$ on $\Sigma_{g', n'}$, we prove the following:

\begin{theorem*}
There exists a homomorphism $i'_\ast$ such that the following diagram commutes:
\begin{equation}
 \xymatrix{
 \mathcal{I}(\Sigma_{g,n}, \mathcal{P}) \ar[r]^{i_\ast} \ar[d]_{\widetilde{e}_Y} &
 \mathcal{I}(\Sigma_{g',n'}, \mathcal{P}')\ar[d]^{\widetilde{e}_X} \\
 \textnormal{Hom} (H_1^{\mathcal{P}}(\Sigma_{g,n}; \mathbb{Z}), \mathbb{Z}) \ar[r]_{i'_\ast}                               & \textnormal{Hom} (H_1^{\mathcal{P}'}(\Sigma_{g', n'}; \mathbb{Z}), \mathbb{Z})} 
\label{eq:1000}
\end{equation}
Here $Y$ is the restriction of $X$ to $\Sigma_{g, n}$.
\end{theorem*}

We also prove that $\widetilde{e}_Y$ is unique in the sense that it is the 
only nontrivial homomorphism such that diagram~(\ref{eq:1000}) commutes. 
A commutative diagram for the 
Chillingworth homomorphism 
$t_{(\Sigma_{g,n}, \mathcal{P})}:\mathcal{I}(\Sigma_{g,n}, \mathcal{P})\rightarrow H_1^{\mathcal{P}}(\Sigma_{g,n}; \mathbb{Z})$ is also obtained.

\noindent{\bf Acknowledgements.}
I would like to thank to Ingrid Irmer and Mustafa Korkmaz for their supervision of this project. This project is part of my Ph.D thesis \cite{Hatice} at Middle East Technical University.


\section{Preliminaries}

In this section, we review some background knowledge and give some preliminary definitions that will be used throughout 
the paper. 

The \textit{mapping class group} $\mathcal{M}(\Sigma_{g,n})$ of 
$\Sigma_{g,n}$ is the group of isotopy 
classes of orientation-preserving diffeomorphisms of $\Sigma_{g,n}$ onto itself 
which fix the boundary components of $\Sigma_{g,n}$ pointwise. 

Throughout this paper, we will be working with representatives of 
mapping classes that fix a neighborhood of the boundary pointwise. 
We will use the notation $f\circ h$ or $fh$ to denote the composition of 
maps, where $h$ is assumed to be applied first.

The subgroup of $\mathcal{M}(\Sigma_{g,1})$ acting trivially on 
$H_1(\Sigma_{g,1};\mathbb{Z})$ is a normal subgroup of $\mathcal{M}(\Sigma_{g,1})$ 
and is called the \textit{Torelli group}. In other words, the Torelli group 
is the kernel of the symplectic representation 
$\rho: \mathcal{M}(\Sigma_{g,1}) \rightarrow \text{Sp}(2g, \mathbb{Z})$. 
It will be denoted by the symbol $\mathcal{I}(\Sigma_{g,1})$. 

\textbf{Winding Number:} If a surface $\Sigma_{g,n}$ has nonempty 
boundary, a nonvanishing vector field $X$ on $\Sigma_{g,n}$ exists. 
By choosing an appropriate parametrisation for a smooth closed curve, it 
can be assumed without loss of generality that the curve has a nonvanishing 
tangent vector at each point of the curve. 

Let us choose a Riemannian metric on $\Sigma_{g,n}$ with which we define a norm 
on $T_x\Sigma_{g,n}$, the tangent space to $\Sigma_{g,n}$ at $x\in\Sigma_{g,n}$, for each $x\in\Sigma_{g,n}$. 

Informally, given 
a nonvanishing vector field $X$, the \textit{winding number $w_X(\gamma)$ of a smooth closed 
oriented curve} $\gamma$ on a surface is defined as the number of rotations 
its tangent vector makes with respect to $X$ as $\gamma$ is traversed once in 
the positive direction \cite{Chillingworth}.

\textbf{The Chillingworth Class:} In \cite{Johnson}, Johnson defined the following homomorphism 
\[
e: \mathcal{I}(\Sigma_{g,1})\rightarrow H^1(\Sigma_{g,1}; \mathbb{Z})
\]
such that $e(f)([\gamma])= w_X(f\gamma)- w_X(\gamma)$.

For $f\in \mathcal{I}(\Sigma_{g,1})$, in Section $5$ of \cite{Johnson}, Johnson 
dualized the class $e(f)$ to a homology class $t(f)$ defined as follows: 
$[\gamma]\cdot t(f)=e(f)[\gamma]$ . The homology class $t(f)$ is called 
the \textit{Chillingworth class} of $f$. In \cite{Johnson}, Johnson proved that $C(\tau(f))= t(f)$ holds for any $f\in\mathcal{I}(\Sigma_{g,1})$, where $\tau$ is the Johnson homomorphism and $C$ is the tensor contraction map.

The Johnson homomorphism 
$\tau:\mathcal{I}(\Sigma_{g,1})\rightarrow \bigwedge^3 H_1(\Sigma_{g,1};\mathbb{Z})$ 
is a surjective homomorphism. 

The tensor contraction map 
$C:\bigwedge^3 H_1(\Sigma_{g,1}; \mathbb{Z})\rightarrow H_1(\Sigma_{g,1};\mathbb{Z})$
is defined as follows: 
\[
 C(x \wedge y \wedge z)= 2[(x \cdot y)z + (y \cdot z)x + (z \cdot x)y],
\]
where $\cdot$ denotes the intersection pairing of homology classes.

\subsection{Subsurface Torelli Groups}

A partitioned surface is a pair $(\Sigma, \mathcal{P})$ consisting 
of a surface $\Sigma$ and a partition $\mathcal{P}$ of the boundary components 
of $\Sigma$. Note that when the genus and the number of boundary components are not important, we use $\Sigma$ to denote the surface. Each element $P_k$ of $\mathcal{P}$ is called a block. 
If each element of the partition contains only 
one boundary component, it is called a \textit{totally separated surface} \cite{Church}. 

For a given embedding $i:\Sigma\hookrightarrow\Sigma_g$, let the connected 
components of $\Sigma_g\setminus\Sigma^{\circ}$ be $\{S_0, S_1, \ldots, S_m\}$ 
and let $P_k$ denote the set of boundary components of $S_k$ for each 
$k\in \{0,\ldots, m\}$. Here, $\Sigma^{\circ}$ denotes the interior of $\Sigma$. 
Consider the partition
\[
\mathcal{P}= \{ P_0, P_1,\ldots, P_m\}.
\] 
Then $i:\Sigma\hookrightarrow\Sigma_g$ is called a capping of $(\Sigma, \mathcal{P})$ (c.f. \cite{Putman1}).

For a partitioned surface $(\Sigma, \mathcal{P})$, in \cite{Putman1} Putman defined the subsurface Torelli group $\mathcal{I}(\Sigma, \mathcal{P})$ to be the subgroup 
$i_\ast^{-1}(\mathcal{I}(\Sigma_g))$ of $\mathcal{M}(\Sigma)$ for any 
capping $i:\Sigma\hookrightarrow\Sigma_g$.

In \cite{Putman1}, Section $3$, a special homology group 
$H_1^{\mathcal{P}}(\Sigma; \mathbb{Z})$ is defined on a partitioned surface 
$(\Sigma, \mathcal{P})$ such that $\mathcal{I}(\Sigma, \mathcal{P})$ 
is the kernel of $\mathcal{M}(\Sigma)\rightarrow \rm{Aut}(H_1^{\mathcal{P}}(\Sigma; \mathbb{Z}))$.

Consider a partition 
\[
\mathcal{P}= \{\{\partial_1^1,\ldots, \partial_{k_1}^1\},\ldots,\{\partial_1^m,\ldots, \partial_{k_m}^m\}\}.
\] 

Suppose the boundary components $\partial_i^j$ are oriented so that 
$\sum_{i,j}[\partial_i^j]=0$ in $H_1(\Sigma;\mathbb{Z})$. Define the 
homology group 
\[
\overline{H}_1^{\mathcal{P}}(\Sigma; \mathbb{Z}):=H_1(\Sigma; \mathbb{Z})/\partial H_1^{\mathcal{P}}(\Sigma; \mathbb{Z}),
\]
where  
\[
\partial H_1^{\mathcal{P}}(\Sigma; \mathbb{Z})=\left\langle ([\partial_1^1]+\ldots+[\partial_{k_1}^1]),\ldots,([\partial_1^m]+\ldots+[\partial_{k_m}^m])\right\rangle\subset H_1(\Sigma;\mathbb{Z}).
\]
  
\begin{definition}[\cite{Putman1}, Section $3.1$] 
Let $(\Sigma, \mathcal{P})$ be a partitioned surface, and let $\mathcal{Q}$ 
denote a set containing one point from each boundary component of $\Sigma$. 
The homology group $H_1^{\mathcal{P}}(\Sigma; \mathbb{Z})$ is defined to be 
the image of the following subgroup of $H_1(\Sigma, \mathcal{Q}; \mathbb{Z})$ 
in $H_1(\Sigma, \mathcal{Q}; \mathbb{Z})/\partial H_1^{\mathcal{P}}(\Sigma; \mathbb{Z})$:
\begin{align} \langle \{[h]\in H_1(\Sigma, \mathcal{Q}; \mathbb{Z}) |& \textrm{ $h$ is either a simple closed curve or a properly embedded arc $a$} \nonumber\\
                                                      & \textrm{ connecting two boundary curves in the same block of $\mathcal{P}$ and with} \nonumber\\
                                     & \textrm{ $\partial a\in \mathcal{Q}$}\}\rangle \nonumber
\end{align}

\end{definition}
 
One can easily see that $\mathcal{M}(\Sigma)$ acts on $H_1^{\mathcal{P}}(\Sigma; \mathbb{Z})$. 
In Theorem $3.3$ of \cite{Putman1}, Putman proves that 
the subsurface Torelli group $\mathcal{I}(\Sigma, \mathcal{P})$ is the 
subgroup of $\mathcal{M}(\Sigma)$ that acts trivially on 
$H_1^{\mathcal{P}}(\Sigma; \mathbb{Z})$.

A $\mathcal{P}$-separating curve on a partitioned surface 
$(\Sigma, \mathcal{P})$ is a simple closed curve $\gamma$ with $[\gamma]=0$ 
in $H_1^{\mathcal{P}}(\Sigma; \mathbb{Z})$. 
A twist about $\mathcal{P}$-bounding pair is defined 
to be $T_{\gamma_1}T_{\gamma_2}^{-1}$, where $\gamma_1$ and $\gamma_2$ are disjoint, 
nonisotopic simple closed curves and $[\gamma_1]=[\gamma_2]$ in 
$H_1^{\mathcal{P}}(\Sigma; \mathbb{Z})$. 
For $g\geq 1$, $\mathcal{I}(\Sigma_{g, n}, \mathcal{P})$ is generated 
by twists about $\mathcal{P}$-separating curves and twists about $\mathcal{P}$-bounding pairs \cite{Putman1}. 

A category $\mathrm{TSur}$ was defined in \cite{Putman1} such that $\mathcal{I}(\Sigma_{g, n}, \mathcal{P})$ is a functor from $\mathrm{TSur}$ to the category of groups and homomorphisms. The objects of $\mathrm{TSur}$ are 
partitioned surfaces $(\Sigma, \mathcal{P})$ and the morphisms from 
$(\Sigma_{g_1,n_1}, \mathcal{P}_1)$ to $(\Sigma_{g_2,n_2}, \mathcal{P}_2)$ 
are exactly those embeddings $i:\Sigma_{g_1,n_1}\hookrightarrow\Sigma_{g_2,n_2}$ 
which induce morphisms 
$i_\ast:\mathcal{I}(\Sigma_{g_1,n_1}, \mathcal{P}_1)\rightarrow \mathcal{I}(\Sigma_{g_2,n_2}, \mathcal{P}_2)$. The embeddings satisfy the following condition: for any 
$\mathcal{P}_1$-separating curve $\gamma$, the curve $i(\gamma)$ must be a 
$\mathcal{P}_2$-separating curve. In this paper, we will use the refinement 
of this category defined by Church in \cite{Church}.

Before giving the definition of the category defined by Church, we need to describe a construction of a minimal totally separated surface $\widehat{\Sigma}$ containing $\Sigma$.
  
\begin{remark}[\cite{Church}]
\label{rmk:totally}
Given a partitioned surface $(\Sigma, \mathcal{P})$, a minimal totally separated surface containing $\Sigma$ can be constructed as follows: For each $P\in \mathcal{P}$ with $|P|=n$, we attach a sphere 
with $n+1$ boundary components to the $n$ boundary components in $P$ of $\Sigma$ 
to obtain $\widehat{\Sigma}$ with a partition $\widehat{\mathcal{P}}$. Each element 
of the partition $\widehat{\mathcal{P}}$ contains only one boundary component. 
\end{remark}

\textbf{Notation:} Given a partitioned surface $(\Sigma, \mathcal{P})$, the partitioned surface 
$(\widehat{\Sigma}, \widehat{\mathcal{P}})$ will denote a minimal totally separated surface containing 
$\Sigma$.

Note that $H_1^{\widehat{\mathcal{P}}}(\widehat{\Sigma}; \mathbb{Z})$ is 
isomorphic to $H_1^P(\Sigma; \mathbb{Z})$.

\begin{definition}[\cite{Church}, Section $2.3$] 
\label{def:cat} 
Objects of the Torelli category $\mathcal{T}\mathrm{Surf}$ are 
partitioned surfaces $(\Sigma, \mathcal{P})$. A morphism from 
$(\Sigma_1, \mathcal{P}_1)$ to $(\Sigma_2, \mathcal{P}_2)$ is 
an embedding $i: \Sigma_1 \hookrightarrow \Sigma_2$ 
satisfying the following conditions:
\begin{itemize}
\item $i$ takes $\mathcal{P}_1$-separating 
curves to $\mathcal{P}_2$-separating curves.

\item $i$ extends to an embedding $\widehat{i}: \widehat{\Sigma}_1\hookrightarrow \widehat{\Sigma}_2$.
\end{itemize}

\end{definition}

In \cite{Church}, given a surface $(\Sigma, \mathcal{P})$ with 
$\mathcal{P}= \{P_0, P_1, \ldots, P_k\}$, Church defined the partitioned Johnson homomorphism $\tau_{(\Sigma, \mathcal{P})}$ 
with image $W_{(\Sigma, \mathcal{P})}$ given in Definition $5.8$ of \cite{Church}. The definition of the partitioned Johnson homomorphism is similar to the definition of the Johnson homomorphism.
Church stated in \cite{Church}, Definition $5.12$, that
$W_{(\Sigma, \mathcal{P})}$ can be considered to be a subspace of 
$\bigwedge^3 H_1^{\widehat{\mathcal{P}}}(\widehat{\Sigma}; \mathbb{Z})\oplus (\mathbb{Z}^k\otimes H_1^{\widehat{\mathcal{P}}}(\widehat{\Sigma}; \mathbb{Z}))$. 
Basis elements of $W_{(\Sigma, \mathcal{P})}$ 
is shown to be $a\wedge b\wedge c$ for the $\bigwedge^3 H_1^{\widehat{\mathcal{P}}}(\widehat{\Sigma}; \mathbb{Z})$ component and as $z_i\wedge x$ 
for $\mathbb{Z}^k\otimes H_1^{\widehat{\mathcal{P}}}(\widehat{\Sigma}; \mathbb{Z})$, where $a, b,c, x\in H_1^{\widehat{\mathcal{P}}}(\widehat{\Sigma}; \mathbb{Z})$ and $z_i$ is the boundary component of $\widehat{\Sigma}$ corresponding to $P_i\in\mathcal{P}$ for each $1\leq i\leq k $.


\section{Results}
In this section, we construct 
a well-defined map $\widetilde{e}_X$ by means of the projective tangent 
bundle. We prove that $\widetilde{e}_X$ and the homomorphism obtained 
by taking the dual of $\widetilde{e}_X(f)$ for any 
$f\in\mathcal{I}(\Sigma, \mathcal{P})$ 
satisfy the naturality property. We define the homomorphism 
from the subsurface Torelli groups to $H_1^{\mathcal{P}}(\Sigma; \mathbb{Z})$ obtained by taking dual of $\widetilde{e}_X(f)$ to be the Chillingworth homomorphism of the subsurface Torelli groups. Moreover, we show that $\widetilde{e}_X$ is 
the unique nontrivial homomorphism satisfying naturality. Finally, we relate 
the Chillingworth classes of the subsurface Torelli groups to the 
partitioned Johnson homomorphism defined by Church.

In this section, if $(\Sigma_1, \mathcal{P}_1)$ and $(\Sigma_2, \mathcal{P}_2)$ are 
partitioned surfaces, then by an embedding $i: (\Sigma_1, \mathcal{P}_1)\hookrightarrow (\Sigma_2, \mathcal{P}_2)$
of partitioned surfaces, we mean a morphism 
$i: (\Sigma_1, \mathcal{P})\rightarrow (\Sigma_2, \mathcal{P})$ of 
$\mathcal{T}\mathrm{Surf}$.


\subsection{Winding Number In The Projective Tangent Bundle}
\label{section:winding}
This section starts with the definition of the projective tangent bundle and we introduce the winding number in the projective tangent bundle.

Let $\Sigma$ be a smooth compact connected oriented surface 
with nonempty boundary. Let us choose a Riemannian metric on $\Sigma$. 
Let $UT(\Sigma)$ be the unit tangent bundle 
of $\Sigma$. Since $\Sigma$ has nonempty boundary, there are  
nonvanishing vector fields on $\Sigma$. Choice of two nonvanishing vector fields 
which are orthogonal to each other gives a parallelization of 
$\Sigma$. The unit tangent bundle $UT(\Sigma)$ is therefore homeomorphic 
to $\Sigma\times S^1$.

By using this unit tangent bundle, the projective tangent bundle $PT(\Sigma)$
is constructed 
as follows: By identifying antipodal points in each fiber $S^1$, we obtain 
a fiber bundle whose fiber is $\mathbb{RP}^1$, which is homeomorphic to $S^1$. 
The projective tangent bundle $PT(\Sigma)$
is also homeomorphic to the product $\Sigma\times S^1$ 
since $\Sigma$ has nonempty boundary. 

Let $\{[\alpha_i]\}_{i\in I}\cup \{[x_j],[y_j]\}_{j\in J}$ be a basis for 
$H_1^{\mathcal{P}}(\Sigma; \mathbb{Z})$. Here, $I$ and $J$ are finite index 
sets, each $\alpha_i$ is an arc, and each $x_j, y_j$ is a simple closed curve. 
We assume that all representatives are smooth. 

In this paper, we always take representatives of mapping classes that 
fix points in a regular neighborhood of each boundary component. Therefore,
$f(\alpha_i)$ and $\alpha_i$ have the same tangent vectors 
on a small neighborhood of the boundary components. We denote 
by $f(\alpha_i)\ast\alpha_i^{-1}$ the closed curve obtained 
by first traversing the arc $f(\alpha_i)$ then $\alpha_i$ with 
the reversed orientation. The resulting closed curve has two nondifferentiable 
points on the boundary of the subsurface. Since $f(\alpha_i)$ and $\alpha_i$ 
have the same tangent vectors at the end points, 
in the projective tangent bundle we can calculate the winding number of 
closed oriented curves having two such nondifferentiable points on the boundary. 
When we concatenate arcs to obtain a closed curve, we will assume that 
the tangent spaces of the arcs at the end points coincide.
 
The winding number in the projective tangent bundle is defined in analogy to 
the winding number in the tangent bundle.
We define winding number in the projective tangent bundle for smooth closed 
oriented curves or for closed oriented curves constructed by concatenating a 
pair of smooth arcs as just described. 

Let us denote the winding number in the projective tangent bundle of a closed oriented curve $\gamma$ 
with respect to a nonvanishing vector field $X$ by $\widetilde{w}_X(\gamma)$. Since 
$S^1$ is a double cover of $\mathbb{RP}^1$, for a smooth closed oriented curve $\gamma$ we have $w_X(\gamma)=\frac{\widetilde{w}_X(\gamma)}{2}$.

\subsection{Construction of $\mathbf{\widetilde{e}_X}$} 
In this section our aim is to define a well-defined map 
$\widetilde{e}_X:\mathcal{I}(\Sigma, \mathcal{P})\rightarrow \text{Hom} 
(H_1^{\mathcal{P}}(\Sigma; \mathbb{Z}), \mathbb{Z})$.

Let X be a nonvanishing vector field on 
a partitioned surface $(\Sigma, \mathcal{P})$ and $f$ be an element of the subsurface 
Torelli group of $(\Sigma, \mathcal{P})$. Choose a set of simple closed curves representing a basis of 
$H_1(\Sigma; \mathbb{Z})$. Assigning an integer to each basis element determines 
a homomorphism from $H_1(\Sigma; \mathbb{Z})$ to $\mathbb{Z}$. This integer is chosen 
to be the total number of times that $X$ rotates relative to $f^{-1}X$ as we traverse 
the basis element. This homomorphism, denoted by $d(X, f^{-1}X)$, is defined in 
\cite{Chillingworth}. By Lemma $4.1$ in \cite{Chillingworth}, we have 
\[
d(X, f^{-1}X)[\gamma] = w_X(f\gamma)- w_X(\gamma),
\] 
for any smooth closed oriented curve $\gamma$. In the projective tangent bundle we get 
\[
d(X, f^{-1}X)[\gamma]=\frac{\widetilde{w}_X(f\gamma)-\widetilde{w}_X(\gamma)}{2}.
\] 
Since $f$ fixes every boundary component of 
$\Sigma$, $d(X, f^{-1}X)[\partial]=0$ for any boundary component $\partial$. Therefore, 
$d(X, f^{-1}X)$ induces a homomorphism $\overline{d}(X, f^{-1}X): \overline{H}_1^{\mathcal{P}}
(\Sigma; \mathbb{Z})\rightarrow\mathbb{Z}$ defined by 
\[
\overline{d}
(X, f^{-1}X)[\gamma]=\frac{\widetilde{w}_X(f\gamma)-\widetilde{w}_X(\gamma)}{2}.
\] 

Now our aim is to get a well-defined map 
\[
\widetilde{d}(X, f^{-1}X): H_1^{\mathcal{P}}(\Sigma; \mathbb{Z})\rightarrow \mathbb{Z}
\]
mapping an element $[\alpha]$ of $H_1^{\mathcal{P}}(\Sigma; \mathbb{Z})$ to the half of 
the number of times that $X$ rotates relative to $f^{-1}X$ in the projective tangent bundle as we traverse $\alpha$. 


For a closed oriented curve $\gamma$, 
we define 
\[
\widetilde{d}(X, f^{-1}X)[\gamma]= \overline{d}(X, f^{-1}X)[\gamma].
\]
Now, let $h$ be a smooth oriented arc whose endpoints are on 
the boundary components of $\Sigma$ contained in the same element of $\mathcal{P}$ 
and let $f\in \mathcal{I}(\Sigma, \mathcal{P})$. Since $f$ fixes all points of 
a regular neighborhood of the boundary components, $h$ and $f(h)$ have the same tangent spaces at the end points and $f(h)\ast h^{-1}$ is a closed oriented curve with two cusps. 
We define 
\[
\widetilde{d}(X, f^{-1}X)[h]:= \frac{\widetilde{w}_X(f(h)\ast h^{-1})}{2}.
\] 

For each $P\in \mathcal{P}$ with $|P|=n$, let us 
attach a sphere with $n+1$ boundary components to the 
$n$ boundary components in $P$ of $\Sigma$ to obtain 
$\widehat{\Sigma}$ with a partition $\widehat{\mathcal{P}}$ as in Remark~\ref{rmk:totally}. 
Thus, $(\widehat{\Sigma},\widehat{\mathcal{P}})$ is totally separated.
Extend $X$ to the obtained larger surface $\widehat{\Sigma}$ so that it is again 
a nonvanishing vector field on $\widehat{\Sigma}$. For simplicity, the extension 
will also be denoted by $X$. Let $h_1$ be a smooth oriented arc in the complement of $\Sigma$ 
whose end points are $\partial h$. Let $\gamma:=h\ast h_1$ 
denote the smooth closed oriented curve obtained by concatenating $h$ and $h_1$. Notice that 
we choose a consistent orientation for $h_1$ to get a closed oriented curve $\gamma$. 
We parametrize $\gamma$ such that its initial and terminal points are on one of the 
boundary components of the subsurface $\Sigma$. Then $f\gamma$ is isotopic to $f(h)\ast h_1$.

\begin{remark}
\label{rmk:note}
The winding number in the projective tangent bundle of the 
concatenation of smooth closed oriented curves 
is equal to the sum of the winding numbers 
of each smooth closed oriented curve if the tangent spaces of the curves at the end points 
are the same. Therefore, we obtain the following equalities:
\begin{eqnarray}
\frac{\widetilde{w}_X(f\gamma)-\widetilde{w}_X(\gamma)}{2} 
		& = & \frac{\widetilde{w}_X(f\gamma\ast\gamma^{-1})}{2} \nonumber \\                              
		& = &  \frac{\widetilde{w}_X(f(h)\ast h_1\ast(h\ast h_1)^{-1})}{2}\nonumber \\                            
		& = &  \frac{\widetilde{w}_X(f(h)\ast h^{-1})}{2}.\nonumber\end{eqnarray}
\end{remark}
One can easily observe that the obtained equality 
\[ 
\frac{\widetilde{w}_X(f\gamma)-\widetilde{w}_X(\gamma)}{2}= \frac{\widetilde{w}_X(f(h)\ast h^{-1})}{2}
\]
does not depend on the choice of the arc representative $h_1$ on $\widehat{\Sigma}\setminus\Sigma^{\circ}$.

\begin{lemma} 
Let $h$ be a smooth oriented arc representing a homology class $[h]$ in 
$H_1^{\mathcal{P}}(\Sigma; \mathbb{Z})$. Then the number 
$\frac{\widetilde{w}_X(f(h)\ast h^{-1})}{2}$ is independent of 
the choice of the representative of $[h]$.
\end{lemma}

\begin{proof}
Let $[h]=[h']$ be in $H_1^{\mathcal{P}}(\Sigma; \mathbb{Z})$. 
Then we have $[h\ast h'^{-1}]= 0$ in $H_1^{\mathcal{P}}(\Sigma; \mathbb{Z})$. 
Since the embedding $(\Sigma, \mathcal{P})\hookrightarrow (\widehat{\Sigma}, \widehat{\mathcal{P}})$ of partitioned surfaces takes $\mathcal{P}$-separating curves to $\widehat{\mathcal{P}}$-separating curves by the first condition 
of Definition~\ref{def:cat}, we get $[h\ast h'^{-1}]= 0$ in $H_1^{\widehat{\mathcal{P}}}
(\widehat{\Sigma}; \mathbb{Z})$. We have 
$[\gamma]=[\gamma']$ by using the following equalities:
\[
[h\ast h'^{-1}]= [h\ast h_1\ast h_1^{-1}\ast h'^{-1}]
= [h\ast h_1]- [h'\ast h_1]=0,
\] 
where $[\gamma]=[h\ast h_1]$ and $[\gamma']=[h'\ast h_1]$. 

Since we have 
\[
\frac{\widetilde{w}_X(f\gamma)-\widetilde{w}_X(\gamma)}{2}
=\frac{\widetilde{w}_X(f\gamma')-\widetilde{w}_X(\gamma')}{2},
\] 
for any smooth homologous simple closed curves $\gamma$ and $\gamma'$ in 
$H_1^{\widehat{\mathcal{P}}}(\widehat{\Sigma}; \mathbb{Z})$, we get 
\[
\frac{\widetilde{w}_X(f(h)\ast h^{-1})}{2}=\frac{\widetilde{w}_X(f(h')\ast h'^{-1})}{2}.
\] 
\end{proof}

\begin{lemma}
The map $\widetilde{d}(X, f^{-1}X): H_1^{\mathcal{P}}(\Sigma, \mathbb{Z})\rightarrow \mathbb{Z}$ is a homomorphism.
\end{lemma}

\begin{proof}
For smooth closed oriented curves $\gamma_1$ and $\gamma_2$ 
by the definition of $d(X, f^{-1}X)$, we have
\[
\widetilde{d}(X, f^{-1}X)[\gamma_1\ast \gamma_2]= \widetilde{d}(X, f^{-1}X)[\gamma_1]+ \widetilde{d}(X, f^{-1}X)[\gamma_2].
\]

Let $h_1$ and $h_2$ be smooth oriented arcs whose endpoints are on 
the boundary components of $\Sigma$ contained in the same element of $\mathcal{P}$ 
and let us assume that the initial point of $h_2$ is the same as the 
terminal point of $h_1$. Let $[h]$ denote the sum of two homology classes $[h_1]$ and $[h_2]$. 
We obtain the following equalities:
\begin{eqnarray}
\widetilde{d}(X, f^{-1}X)[h_1]+ \widetilde{d}(X, f^{-1}X)[h_2]
& = & \frac{\widetilde{w}_X(f(h_1)\ast h_1^{-1})}{2}+ \frac{\widetilde{w}_X(f(h_2)\ast h_2^{-1})}{2}\nonumber \\
& = & \frac{\widetilde{w}_X(h_1^{-1}\ast f(h_1))}{2}+ \frac{\widetilde{w}_X(f(h_2)\ast h_2^{-1})}{2}\nonumber \\
& = & \frac{\widetilde{w}_X(h_1^{-1}\ast f(h_1)\ast f(h_2)\ast h_2^{-1})}{2} \nonumber \\
& = & \frac{\widetilde{w}_X(f(h_1)\ast f(h_2)\ast h_2^{-1}\ast h_1^{-1})}{2} \nonumber \\
& = & \frac{\widetilde{w}_X(f(h_1\ast h_2)\ast (h_1\ast h_2)^{-1})}{2} \nonumber \\
& = & \widetilde{d}(X, f^{-1}X)[h_1\ast h_2] \nonumber \\
& = & \widetilde{d}(X, f^{-1}X)[h]. \nonumber
\end{eqnarray}
Now let $[h']$ denote a homology class whose representatives are arcs. Let $\gamma$ be a smooth oriented arc whose homology class $[\gamma]$ is the sum of $[h']$ and a homology class $[\alpha]$ with closed curve representatives. As in the previous paragraph of Remark~\ref{rmk:note}, we can obtain a smooth closed oriented curve $\alpha'$ by concatenating $h'$ with a smooth oriented arc in the complement of $\Sigma$. Hence, we have
\begin{eqnarray}
\widetilde{d}(X, f^{-1}X)[h']+ \widetilde{d}(X, f^{-1}X)[\alpha]
& = & \frac{\widetilde{w}_X(f(h')\ast h'^{-1})}{2} + \frac{\widetilde{w}_X(f\alpha) - \widetilde{w}_X (\alpha)}{2} \nonumber \\
& = & \frac{\widetilde{w}_X(f\alpha') - \widetilde{w}_X (\alpha')}{2} +\frac{\widetilde{w}_X(f\alpha) - \widetilde{w}_X (\alpha)}{2} \nonumber \\
& = & \widetilde{d}(X, f^{-1}X) [\alpha'\ast \alpha] \nonumber \\
& = & \widetilde{d}(X, f^{-1}X) [h'\ast \alpha] \nonumber\quad \text{ (by Remark~\ref{rmk:note})}\\
& = & \widetilde{d}(X, f^{-1}X) [\gamma]. \nonumber
\end{eqnarray}
\end{proof}

\begin{definition} 
The map $\widetilde{e}_X:\mathcal{I}(\Sigma, \mathcal{P})\rightarrow \textnormal{Hom}  
(H_1^{\mathcal{P}}(\Sigma; \mathbb{Z}), \mathbb{Z})$ is defined to be 
$\widetilde{e}(f):= \widetilde{d}(X, f^{-1}X)$. More explicitly, it is defined as follows:

If $[\gamma]$ has a smooth closed curve representative $\gamma$,
$$\widetilde{e}_X(f)[\gamma]:= \frac{\widetilde{w}_X(f\gamma)-\widetilde{w}_X(\gamma)}{2}.$$

If $h$ is a smooth oriented arc representing a homology class $[h]$ in $H_1^{\mathcal{P}}(\Sigma; \mathbb{Z})$,
$$\widetilde{e}_X(f)[h]:= \frac{\widetilde{w}_X(f(h)\ast h^{-1})}{2}.$$
\end{definition}

\begin{lemma}
The map $\widetilde{e}_X:\mathcal{I}(\Sigma, \mathcal{P})
\rightarrow \textnormal{Hom} (H_1^{\mathcal{P}}(\Sigma; \mathbb{Z}), \mathbb{Z})$ 
is a homomorphism.
\end{lemma}

\begin{proof} By Lemma $5B$ of \cite{Johnson}, it is easy to see that $\widetilde{e}_X(fg)[\gamma]= \widetilde{e}_X(f)[\gamma]+ \widetilde{e}_X(g)[\gamma]$ for a smooth closed oriented curve $\gamma$.

For a smooth oriented arc $\alpha_i$,
\begin{eqnarray}\widetilde{e}_X(fg)[\alpha_i]& = &\frac{\widetilde{w}_X(fg(\alpha_i)\ast\alpha^{-1}_i)}{2}\nonumber\\
                                             & = &\frac{\widetilde{w}_X(f(g\alpha_i)\ast g(\alpha^{-1}_i)\ast g(\alpha_i)\ast \alpha^{-1}_i)}{2}\nonumber\\
                                             & = &\frac{\widetilde{w}_X(f(g\alpha_i)\ast g(\alpha^{-1}_i))}{2}+\frac{\widetilde{w}_X(g(\alpha_i)\ast\alpha^{-1}_i)}{2}\nonumber\\
                                             & = &\widetilde{e}_X(f)[g(\alpha_i)]+\widetilde{e}_X(g)[\alpha_i].
                                                                                              \nonumber \end{eqnarray}
Since $g\in\mathcal{I}(\Sigma, \mathcal{P})$, $g(\alpha_i)$ and $\alpha_i$ represent the same element of $H_1^{\mathcal{P}}(\Sigma; \mathbb{Z})$. Hence we get 
\[
 \widetilde{e}_X(fg)=\widetilde{e}_X(f)+\widetilde{e}_X(g).
\]
\end{proof}

Notice that $\widetilde{e}_X$ depends on the choice of the 
nonvanishing vector field $X$.


\subsection{Symplectic Basis for $\mathbf{H_1^{\mathcal{P}}(\Sigma; \mathbb{Z})}$}
\label{section:symplecticbasis}
In this section, we introduce a symplectic basis for 
$H_1^{\mathcal{P}}(\Sigma; \mathbb{Z})$. 
 
Let $(\Sigma, \mathcal{P})$ be a partitioned surface of genus $g$ 
with the partition  
\[
\mathcal{P} = \{\{\partial_1^1,
\ldots, \partial_{k_1}^1\},\ldots,\{\partial_1^m,\ldots, \partial_{k_m}^m\}\}.
\] 
Let $\mathcal{Q}$ be a subset of the boundary $\partial \Sigma$
containing exactly one point from each boundary component. 

Let us choose a set of simple closed curves 
$\{x_1, y_1, x_2, y_2, \dots, x_g, y_g\}$ on $\Sigma$ satisfying 
\begin{itemize}
	\item  $x_i\cap x_j= \emptyset, \quad x_i\cap y_j= \emptyset, \quad y_i\cap y_j= \emptyset$ for $i\neq j,$
	\item $x_i$ intersects $y_i$ transversely at one point, and 
	\item under the filling map 
\[
H_1(\Sigma; \mathbb{Z})\rightarrow H_1(\overline{\Sigma}; \mathbb{Z})
\] 
$\{[x_i], [y_i]\ |\ i= 1, \ldots, g\}$ maps to a symplectic basis of 
$H_1(\overline{\Sigma}; \mathbb{Z})$. Here, $\overline{\Sigma}$ denotes the 
closed surface obtained by gluing a disc along each boundary component and the filling map is induced by inclusion.
\end{itemize}

For each $l = 1, 2, \ldots, m$, choose oriented arcs $h_j^l$ connecting 
$\partial_j^l\cap \mathcal{Q}$ to $\partial_{j+1}^l\cap\mathcal{Q}$ for 
$j=1, 2, \ldots, k_l-1$ such that 

\begin{itemize}
\item $h_j^l$ are disjoint from $x_i, y_i$,

\item $h_j^l$ are pairwise disjoint except perhaps at endpoints,

\item each $h_j^l$ is oriented so that the algebraic 
intersection number of the homology classes $[h_j^l]$ 
and $[\partial_1^l+ \cdots + \partial_{j}^l]$ is $1$,
where the orientations of the boundary components are induced from the orientation of the surface.
\end{itemize}

The union of the sets 
\begin{itemize}
\item $\{[x_1], [y_1], \ldots, [x_g], [y_g]\}$,

\item $\{[h_1^1], [h_2^1], \ldots, [h_{k_1-1}^1], [h_1^2], \ldots, 
[h_{k_2-1}^2], \ldots, [h_1^m],\ldots, [h_{k_m-1}^m]\},$

\item $\{\ [\partial_1^1], [\partial_1^1+ \partial_2^1], \ldots, [\partial_1^1+ \cdots + \partial_{k_1-1}^1], [\partial_1^2], \ldots, [\partial_1^2+ \cdots + \partial_{k_2-1}^2], \ldots, [\partial_1^m], 
\ldots,$ $[\partial_1^m+ \cdots + \partial_{k_m-1}^m]\}$
\end{itemize}
is a basis $\mathscr{B}$ of $H_1^{\mathcal{P}}(\Sigma; \mathbb{Z})$.  
   
In this basis, $\{x_{i}, y_{i}\}$ are closed curves, the $\{h_{i_l}^l\}$s are 
arcs, and $\{\partial_{i_l}^l\}$s are boundary curves as shown in Figure \ref{fig:surface}.

This basis $\mathscr{B}$ has the following properties:
\begin{itemize}
	\item $\widehat{i}([x_{i}], [x_{j}]) = \widehat{i}([y_{i}], [y_{j}]) = 0, \qquad \widehat{i}([x_{i}], [y_{j}]) =\delta_{ij},\quad \text{for all } 1 \le i, j \le g,$
	\item $\widehat{i}([h_i^l], [\partial_1^l+ \cdots + \partial_j^l]) = \delta_{ij},\quad \text{for all } 1 \le i, j \le k_l-1, 1\le l \le m,$
	\item $\widehat{i}([h_j^l], [x_{i}]) = \widehat{i}([h_j^l], [y_{i}]) = 0, 
	\text{for all } 1\le i\le g, 1\le j\le k_l-1, 1\le l \le m,$ 
	\item $\widehat{i}([\partial_1^l+ \cdots + \partial_{j}^l], [x_{i}]) 
	= \widehat{i}([\partial_1^l+ \cdots + \partial_{j}^l], [y_{i}])= 0, 
	\text{for all } 1\le i\le g, 1\le j\le k_l-1, 1\le l \le m.$
\end{itemize}
Here, $\delta_{ij}$ denotes the Kronecker delta and $\widehat{i}(\cdot, \cdot)$ 
denotes the algebraic intersection number. Note that although the endpoints of 
the representatives of homology basis elements $[h_{j}^l]$ coincide with 
$[h_{j+1}^l]$ on $\partial\Sigma$, we define the algebraic intersection of 
arcs  $\widehat{i}([h_{j}^{l}], [h_{j+1}^{l}])$ to be 0 for all 
$1\le j\le k_l-1, 1\le l \le m$. 

\begin{figure}[htb] 
 \centering
  \includegraphics[width=8.0cm]{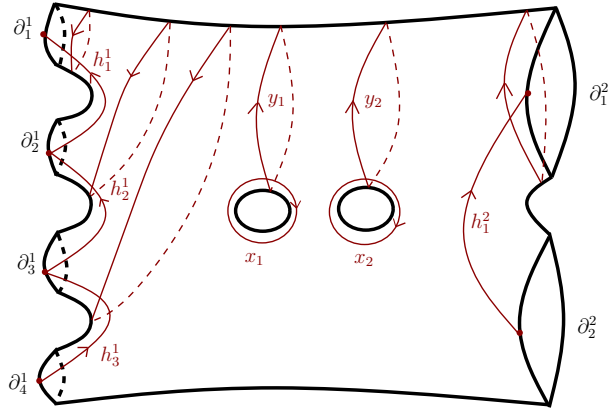}
   \vspace{1.em}
    \caption{An example illustrating homology basis elements of $H_1^{\mathcal{P}}(\Sigma_{2,6}; \mathbb{Z})$, where $\mathcal{P}= \{\{\partial_1^1, \partial_2^1, \partial_3^1, \partial_4^1\}, \{\partial_1^2, \partial_2^2\}\}$.}
       \label{fig:surface}
\end{figure}

We now define the dual of a homology class of 
$H_1^{\mathcal{P}}(\Sigma; \mathbb{Z})$ by using 
this intersection form. Note that the intersection form $\widehat{i}$
is nondegenerate. Therefore the map 
\[ 
D: H_1^{\mathcal{P}}
(\Sigma; \mathbb{Z})\rightarrow \text{Hom} (H_1^{\mathcal{P}}
(\Sigma; \mathbb{Z}), \mathbb{Z})
\]
sending $[x]\in H_1^{\mathcal{P}}(\Sigma; \mathbb{Z})$ to $\widehat{i}(\cdot, [x])$ is an isomorphism.

 
\subsection{Naturality and Uniqueness of $\mathbf{\widetilde{e}_X}$}
\label{section:naturality1}
In this section, we show that $\widetilde{e}_X$ is natural and that it is the unique nontrivial natural
homomorphism from $\mathcal{I}(\Sigma, \mathcal{P})$ to $\text{Hom} (H_1^{\mathcal{P}}(\Sigma; \mathbb{Z}), \mathbb{Z})$.

\begin{remark}
\label{rmk:homology}
Suppose that $(\Sigma,\mathcal{P})$ is a totally separated surface with 
boundary components $z_1,z_2,\ldots,z_n$, so that 
$\mathcal{P}= \{\{z_1\}, \ldots, \{z_n\}\}$. Suppose also that $\Sigma'$ is a 
partitioned surface with a partition $\mathcal{P}'$ such that
there is an embedding $(\Sigma, \mathcal{P}) \hookrightarrow (\Sigma', \mathcal{P}')$ of partitioned surfaces. For $1\leq j\leq n$, let $V_j$ be a connected component 
of $\Sigma'\setminus \Sigma^{\circ}$ containing $z_j$ as a boundary component and let
$\mathcal{P}_j$ be the partition of the boundary of $V_j$ consisting of 
 $\{ z_j\}$ and a subset of $\mathcal P'$. 
By identifying  $H_1^{\mathcal{P}_j}(V_j; \mathbb{Z})$ and $H_1^{\mathcal{P}}(\Sigma; \mathbb{Z})$ with their images in $H_1^{\mathcal{P}'}(\Sigma'; \mathbb{Z})$, we can write 
\[
H_1^{\mathcal{P}'}(\Sigma'; \mathbb{Z})= H_1^{\mathcal{P}}(\Sigma; \mathbb{Z})\oplus H_1^{\mathcal{P}_1}(V_1; \mathbb{Z})\oplus\cdots\oplus H_1^{\mathcal{P}_n}(V_n; \mathbb{Z}).
\]
\end{remark}

If $\Sigma$ is totally separated with the partition 
$\mathcal{P}$ and if 
$i: (\Sigma, \mathcal{P}) \hookrightarrow (\Sigma', \mathcal{P}')$ is an 
embedding of partitioned surfaces, then there is a natural projection 
\[
r_*: H_1^{\mathcal{P}'}(\Sigma'; \mathbb{Z})\to  H_1^{\mathcal{P}}(\Sigma; \mathbb{Z})
\]
which gives a natural homomorphism
\[ 
r^* : \textnormal{Hom} (H_1^{\mathcal{P}}(\Sigma; \mathbb{Z}), \mathbb{Z}) 
\to  \textnormal{Hom} (H_1^{\mathcal{P}'}(\Sigma'; \mathbb{Z}), \mathbb{Z}).
\]



\begin{proposition}
\label{prop:natural}
Let $\Sigma$ be a totally separated surface with the partition 
$\mathcal{P}$ and let 
$i: (\Sigma, \mathcal{P}) \hookrightarrow (\Sigma', \mathcal{P}')$
be an embedding of partitioned surfaces. Let $X$ be a nonvanishing vector field 
on $\Sigma'$ and let $Y$ denote the restriction of $X$ to $\Sigma$.
Then the homomorphism $\widetilde{e}_Y$ is natural in the sense that the diagram
\begin{equation}
\xymatrix{
 \mathcal{I}(\Sigma, \mathcal{P}) \ar[r]^{i_\ast} \ar[d]_{\widetilde{e}_Y} &
                                        \mathcal{I}(\Sigma', \mathcal{P}')\ar[d]^{\widetilde{e}_X} \\
 \textnormal{Hom} (H_1^{\mathcal{P}}(\Sigma; \mathbb{Z}), \mathbb{Z}) \ar[r]_{r^*}                          & \textnormal{Hom} (H_1^{\mathcal{P}'}(\Sigma'; \mathbb{Z}), \mathbb{Z})} 
\label{eq:01}
\end{equation}
commutes.
\end{proposition}

\begin{proof} 
Let $f\in\mathcal{I}(\Sigma, \mathcal{P})$, and let $i_\ast(f)=\widetilde{f}$.
Thus (the class of) the diffeomorphism $\widetilde{f}$ is equal to $f$ on $\Sigma$ and 
is the identity on the complement $\Sigma'\setminus\Sigma$. 
We show that $r^*(\widetilde{e}_Y(f))= \widetilde{e}_X(\widetilde{f})$.

Let $\gamma$ be a smooth oriented simple closed curve in $\Sigma$ representing a basis element of $H_1^{\mathcal{P}}(\Sigma; \mathbb{Z})$. Then, we have 
\[
r^*(\widetilde{e}_{Y}(f))[\gamma]=\widetilde{e}_{Y}(f) (r_* [\gamma])
=\widetilde{e}_{Y}(f) [\gamma]
= \frac{\widetilde{w}_{Y}(f\gamma)-\widetilde{w}_{Y}(\gamma)}{2}
\] 
and 
\[
\widetilde{e}_X(\widetilde{f})[\gamma]= \frac{\widetilde{w}_X(\widetilde{f}\gamma)-\widetilde{w}_X(\gamma)}{2}= \frac{\widetilde{w}_X(f\gamma)-\widetilde{w}_X(\gamma)}{2}.
\] 
Since $Y$ is the restriction of $X$ to $\Sigma$, we have 
$r^*(\widetilde{e}_{Y}(f))[\gamma]=\widetilde{e}_X(\widetilde{f})[\gamma].$

Now let $\gamma'$ be a smooth closed oriented curve or smooth oriented arc in some $V_j$ representing a homology basis element
in $H_1^{\mathcal{P}_j}(V_j; \mathbb{Z})$. In this case,
$r^*(\widetilde{e}_{Y}(f))[\gamma']= \widetilde{e}_{Y}(f) (r_*([\gamma']))=0$ 
because $r_*([\gamma'])=0$. Since $f(\gamma')=\gamma'$, we have 
$$
\widetilde{e}_X(\widetilde{f})[\gamma']= \frac{\widetilde{w}_X(\widetilde{f}\gamma')-\widetilde{w}_X(\gamma')}{2}= \frac{\widetilde{w}_X(\gamma')-\widetilde{w}_X(\gamma')}{2}= 0.
$$

Since $H_1^{\mathcal{P}'}(\Sigma'; \mathbb{Z})$ is the direct sum of 
$H_1^{\mathcal{P}}(\Sigma; \mathbb{Z})$ and $H_1^{\mathcal{P}_j}(V_j; \mathbb{Z})$,
it follows that $r^*(\widetilde{e}_{Y}(f))=\widetilde{e}_X(\widetilde{f})$ for every $f$ in
$\mathcal{I}(\Sigma, \mathcal{P})$, and hence $r^* \widetilde{e}_{Y}=\widetilde{e}_X i_*$.
\end{proof}

Suppose now that $\Sigma$ is any surface with a partition 
$\mathcal{P}=\{P_1,P_2,\ldots,P_n\}$, $|P_l|=n_l$, 
and  that $\widehat{\Sigma}$ is the totally separated surface, with 
the partition $\widehat{\mathcal{P}}$, obtained by gluing a sphere 
$S_l$ with $n_l+1$ holes along the boundary components in $P_l$, 
i.e. the minimal totally separated surface containing $\Sigma$ (c.f. Remark~\ref{rmk:totally}).
For an $l=1,2,\ldots,n$, suppose that 
$P_l=\{ \partial^l_1,\partial^l_2,\ldots,\partial^l_{n_l}\}$. For each $j=1,2,\ldots,{n_l}-1$,
choose smooth arcs $k^l_j$ on the complement $\widehat{\Sigma}\setminus\Sigma^{\circ}$ connecting $\mathcal Q\cap\partial_j^l$ to 
$\mathcal{Q}\cap\partial^l_{j+1}$. Here, $k^l_j$ are pairwise disjoint except perhaps at endpoints. 
Let us orient each $k^l_j$ so that concatenation $h^l_j\ast k^l_j$ is a smooth closed oriented 
curve in $\widehat{\Sigma}$, where $[h^l_j]$ is an element of the 
basis $\mathscr{B}$ defined in Subsection~\ref{section:symplecticbasis}. Let 
$\mathcal{P}_l= \{P_l, \{z_l\}\}$ be the partition of the boundary of $S^l$, 
where $z_l$ is the boundary component of $\widehat{\Sigma}$.
Then $K_l= \{[k^l_j]\}$ is a set of basis elements with arc representatives 
of $H_1^{\mathcal{P}_l}(S_l; \mathbb{Z})$. 
Let $\mathscr{K}$ denote the union $K_1\cup K_2\cup \cdots \cup K_n$.

Let us fix the symplectic basis $\mathscr{B}$ 
of $H_1^{\mathcal{P}}(\Sigma; \mathbb{Z})$ defined as in Subsection~\ref{section:symplecticbasis}. 

We then have an isomorphism 
\[ 
\psi_{\mathscr{K}}: H_1^{\mathcal{P}}(\Sigma; \mathbb{Z})\rightarrow H_1^{\widehat{\mathcal{P}}}(\widehat{\Sigma}; \mathbb{Z})
\]
by mapping basis elements with closed curve representatives to themselves 
and $[h^l_{j}]$ to $[h^l_j\ast k^l_j]$. 

By using $\psi_{\mathscr{K}}$, we get the isomorphism
\[
\psi^*_{\mathscr{K}}:\textnormal{Hom}(H_1^{\widehat{\mathcal{P}}}(\widehat{\Sigma}; \mathbb{Z}), \mathbb{Z})\rightarrow \textnormal{Hom}(H_1^{\mathcal{P}}(\Sigma; \mathbb{Z}), \mathbb{Z})
\]
defined to be $\psi^*_{\mathscr{K}}(\chi)= \chi\circ \psi_{\mathscr{K}}$ for any $\chi\in \textnormal{Hom}(H_1^{\widehat{\mathcal{P}}}(\widehat{\Sigma}; \mathbb{Z}), \mathbb{Z})$.

\begin{proposition}
\label{prop:commut}
Let $(\Sigma, \mathcal{P})$ be a partitioned surface and let 
$(\widehat{\Sigma}, \widehat{\mathcal{P}})$ be a minimal totally separated surface
containing $\Sigma$. Let  
$i: (\Sigma, \mathcal{P}) \hookrightarrow (\widehat{\Sigma}, \widehat{\mathcal{P}})$
be an inclusion coming from an embedding of partitioned surfaces. Let $X$ be a nonvanishing vector field 
on $\widehat{\Sigma}$ and let $Y$ denote the restriction of $X$ to $\Sigma$.
Then the homomorphism $\widetilde{e}_Y$ is natural in the sense that the diagram
\begin{equation}
\xymatrix{
 \mathcal{I}(\Sigma, \mathcal{P}) \ar[r]^{i_\ast} \ar[d]_{\widetilde{e}_Y} &
 \mathcal{I}(\widehat{\Sigma}, \widehat{\mathcal{P}})\ar[d]^{\widetilde{e}_X} \\
\textnormal{Hom} (H_1^{\mathcal{P}}(\Sigma; \mathbb{Z}), \mathbb{Z})    & 
\textnormal{Hom} (H_1^{\widehat{\mathcal{P}}}(\widehat{\Sigma}; \mathbb{Z}), \mathbb{Z})\ar[l]_{\psi^{*}_{\mathscr{K}}}} 
\label{eq:0111}
\end{equation}
commutes.
\end{proposition}

\begin{proof}
Let $f\in\mathcal{I}(\Sigma, \mathcal{P})$, and let $i_\ast(f)=\widetilde{f}$.
Thus $\widetilde{f}$ is equal to $f$ on $\Sigma$ and 
is the identity on the complement $\widehat{\Sigma}\setminus\Sigma$. 
We show that $\widetilde{e}_Y(f)= \psi^*_{\mathscr{K}}\widetilde{e}_X(\widetilde{f})$.

For any homology basis element $[\gamma]\in H_1^{\mathcal{P}}(\Sigma; \mathbb{Z})$ 
with a smooth closed oriented curve representative $\gamma$, we have 
\[
\widetilde{e}_Y(f)[\gamma]= \frac{\widetilde{w}_Y(f\gamma)-\widetilde{w}_Y(\gamma)}{2}
\]
and 
\begin{eqnarray}
\psi^*_{\mathscr{K}}\widetilde{e}_{X}(\widetilde{f})([\gamma])
& = & \widetilde{e}_{X}(\widetilde{f})(\psi_{\mathscr{K}}[\gamma])\nonumber \\
& = & \widetilde{e}_{X}(\widetilde{f})[\gamma] \nonumber \\
& = & \frac{\widetilde{w}_{X}(\widetilde{f}\gamma)-\widetilde{w}_{X}(\gamma)}{2} \nonumber \\
& = & \frac{\widetilde{w}_{X}(f\gamma)-\widetilde{w}_{X}(\gamma)}{2}. \nonumber
\end{eqnarray}
Since $X=Y$ on $\Sigma$, we get the desired equality.
 
For any homology basis element $[h^l_j]\in H_1^{\mathcal{P}}(\Sigma; \mathbb{Z})$ 
with a smooth oriented arc representative $h^l_j$, we have
\[
\widetilde{e}_Y(f)[h^l_j] = \frac{w_Y(f(h^l_j)\ast (h^{l}_j)^{-1})}{2}
\]
and 
\begin{eqnarray}
\psi^*_{\mathscr{K}}\widetilde{e}_{X}(\widetilde{f})[h^l_j]
& = &  \widetilde{e}_X(\widetilde{f})(\psi_{\mathscr{K}}[h^l_j]) \nonumber \\
& = &  \widetilde{e}_X(\widetilde{f})[h^l_j\ast k^l_j] \nonumber \\
& = & \frac{\widetilde{w}_{X}(\widetilde{f}(h^l_j\ast k^l_j))-\widetilde{w}_{X}(h^l_j\ast k^l_j)}{2}.\nonumber
\end{eqnarray}
Since we are working in the projective tangent bundle 
and we assume that representatives of mapping classes fix a regular neighborhood 
of the boundary components, we get
\begin{eqnarray}
\psi^*_{\mathscr{K}}\widetilde{e}_{X}(\widetilde{f})[h^l_j] 
& = &\frac{\widetilde{w}_{X}(\widetilde{f}(h^l_j\ast k^l_j) \ast (h^l_j\ast k^l_j)^{-1})}{2}\nonumber \\
& = &\frac{\widetilde{w}_{X}(f(h^l_j)\ast k^{l}_j\ast (k^{l}_j)^{-1}\ast (h^l_j)^{-1})}{2} \nonumber\\
& = &\frac{\widetilde{w}_{X}(f(h^l_j)\ast (h^l_j)^{-1})}{2} \nonumber \\
& = &\frac{\widetilde{w}_Y(f(h^l_j)\ast (h^l_j)^{-1})}{2}.\nonumber
\end{eqnarray}

Therefore, we obtain the equality 
$\widetilde{e}_Y= \psi^*_{\mathscr{K}}\widetilde{e}_X i_*$. 
This concludes the proof.
\end{proof}

Note that commutativity of diagram~(\ref{eq:0111}) does not depend on the choice of basis $\{[k^{l}_j]\}\in H_1^{\mathcal{P}_l}(S_l; \mathbb{Z})$.

Proposition~\ref{prop:natural} and Proposition~\ref{prop:commut} imply the following theorem.

\begin{theorem}
\label{thm:natural}
Let $(\Sigma, \mathcal{P})$ and $(\Sigma', \mathcal{P}')$ be partitioned surfaces and
$i: (\Sigma, \mathcal{P}) \hookrightarrow (\Sigma', \mathcal{P}')$
be an embedding of partitioned surfaces. Let $X$ be a nonvanishing vector field 
on $\Sigma'$ and let $Y$ denote the restriction of $X$ to $\Sigma$. Then 
there exists a homomorphism $i'_\ast$ such that the homomorphism 
$\widetilde{e}_Y$ is natural in the sense that the diagram
\begin{equation}
\xymatrix{
 \mathcal{I}(\Sigma, \mathcal{P}) \ar[r]^{i_\ast} \ar[d]_{\widetilde{e}_Y} &
                                        \mathcal{I}(\Sigma', \mathcal{P}')\ar[d]^{\widetilde{e}_X} \\
 \textnormal{Hom} (H_1^{\mathcal{P}}(\Sigma; \mathbb{Z}), \mathbb{Z}) \ar[r]_{i'_\ast}                          & \textnormal{Hom} (H_1^{\mathcal{P}'}(\Sigma'; \mathbb{Z}), \mathbb{Z})} 
\label{eq:1}
\end{equation}
commutes.
\end{theorem}

\begin{proof} 
Let $\mathcal{P}=\{P_1,P_2,\ldots,P_n\}$, $|P_l|=n_l$ be the partition on $\Sigma$. 
For an $l=1,2,\ldots,n$, suppose that 
$P_l=\{ \partial^l_1,\partial^l_2,\ldots,\partial^l_{n_l}\}$. For each $j=1,2,\ldots,{{n_l}-1}$,
choose smooth oriented simple arcs $k^l_j$ on the complement $\Sigma'\setminus \Sigma^{\circ}$ 
connecting $\mathcal Q\cap\partial^l_j$ to
$\mathcal{Q}\cap\partial^l_{j+1}$. Here, $k^l_j$ are pairwise disjoint except perhaps 
at endpoints. We consider a closed tubular neighbourhood of the union
$\partial^l_1\cup\partial^l_2\cup\cdots\cup\partial^l_{n_l}\cup k^l_1\cup\cdots k^l_{n_l-1}$. 
This tubular neighbourhood is homeomorphic to a sphere $S_l$ with $n_l+1$ holes.
Let us consider now a minimal totally separated surface $(\widehat{\Sigma}, \widehat{\mathcal{P}})$ containing $\Sigma$ and all $S_l$ as a subsurface.

Let us fix bases $\mathscr{B}$ and $\mathscr{K}$
as in Proposition~\ref{prop:commut}.

Consider the composition of the embedding 
$\widehat{j}: (\Sigma, \mathcal{P})\hookrightarrow(\widehat{\Sigma}, \widehat{\mathcal{P}})$ 
of partitioned surfaces with the embedding 
$j': (\widehat{\Sigma}, \widehat{\mathcal{P}})\hookrightarrow (\Sigma', \mathcal{P}')$ 
of partitioned surfaces. Let $\widehat{Y}$ denote the restriction of $X$ to $\widehat{\Sigma}$.
After showing that both diagrams in (\ref{eq:3}) are commutative, our proof 
will be complete.
\begin{equation}
\xymatrix{
 \mathcal{I}(\Sigma, \mathcal{P}) \ar[r]^{\widehat{j}_\ast} \ar[d]_{\widetilde{e}_Y} & \mathcal{I}(\widehat{\Sigma}, \widehat{\mathcal{P}})\ar[r]^{j'_\ast} \ar[d]^{\widetilde{e}_{\widehat{Y}}} 
                                                                 & \mathcal{I}(\Sigma',\mathcal{P}')\ar[d]^{\widetilde{e}_X}\\
\text{Hom}(H_1^{\mathcal{P}}(\Sigma; \mathbb{Z}), \mathbb{Z})\ar[r]_{(\psi^*_{\mathscr{K}})^{-1}}        & \text{Hom}(H_1^{\widehat{P}}(\widehat{\Sigma}; \mathbb{Z}), \mathbb{Z})\ar[r]_{r^*}  & \text{Hom}(H_1^{\mathcal{P}'}(\Sigma'; \mathbb{Z}), \mathbb{Z}) }
\label{eq:3}                                                                                                                                                              
\end{equation}
Proposition~\ref{prop:natural} implies the commutativity of the 
right-hand side in diagram~(\ref{eq:3}). Proposition~\ref{prop:commut} gives the commutativity of the left-hand side in diagram~(\ref{eq:3}).

\end{proof}

\begin{remark} 
Theorem~\ref{thm:natural} remains true for any capping
$i:(\Sigma, \mathcal{P})\hookrightarrow \Sigma_g$ under the condition that 
the chosen vector field $X$ on $\Sigma_{g}$ has only one singularity in the 
complement of $\widehat{\Sigma}$.
\end{remark}

\begin{proposition}
\label{pro:unique} 
The homomorphism $\widetilde{e}_Y$ is unique in the sense that it is the 
only nontrivial homomorphism from $\mathcal{I}(\Sigma, \mathcal{P})$ to $\textnormal{Hom} (H_1^{\mathcal{P}}(\Sigma; \mathbb{Z}), \mathbb{Z})$ such that diagram~(\ref{eq:1}) commutes.
\end{proposition}

\begin{proof} 
Let $\widehat{\Sigma}$ be a totally separated surface obtained 
from $\Sigma$ as in Remark~\ref{rmk:totally}. Since an embedding 
$i:(\Sigma, \mathcal{P})\hookrightarrow (\Sigma', \mathcal{P}')$ of partitioned 
surfaces can be considered to be the composition of the two embeddings 
$(\Sigma, \mathcal{P})\hookrightarrow (\widehat{\Sigma}, \widehat{\mathcal{P}})\hookrightarrow (\Sigma', \mathcal{P}')$ of partitioned surfaces, we will consider the diagram~(\ref{eq:3}).

First, we show the uniqueness of 
$\widetilde{e}_{\widehat{Y}}$ such that the right side of diagram (\ref{eq:3}) 
is commutative. The second step will be to show the uniqueness of $\widetilde{e}_Y$ 
such that the left side of diagram (\ref{eq:3}) is commutative. This will finish our proof. 

Now let us consider the embedding $(\widehat{\Sigma}, \widehat{\mathcal{P}})\hookrightarrow (\Sigma', \mathcal{P}')$ of partitioned surfaces. We have $r^*\circ \widetilde{e}_{\widehat{Y}}= \widetilde{e}_X\circ j'_{\ast}$. 
Let us assume that there is another homomorphism $G:\mathcal{I}(\widehat{\Sigma}, \widehat{\mathcal{P}})\rightarrow \text{Hom}(H_1^{\widehat{\mathcal{P}}}(\widehat{\Sigma}; \mathbb{Z}), \mathbb{Z})$ 
satisfying the naturality condition, $r^*\circ G= \widetilde{e}_X\circ j'_{\ast}$. 
Our aim is to show that $\widetilde{e}_{\widehat{Y}}= G$, hence proving the 
proposition in this case. Since both $G$ and $\widetilde{e}_{\widehat{Y}}$ 
satisfy the naturality condition, we get $r^*\circ \widetilde{e}_{\widehat{Y}}= r^*\circ G$. 
Since $r_{\ast}$ is onto, $r^*$ is injective, which implies that $\widetilde{e}_{\widehat{Y}}= G$. 

Now consider the embedding $(\Sigma, \mathcal{P})\hookrightarrow (\widehat{\Sigma}, \widehat{P})$ 
of partitioned surfaces for the second part of the proof. We need to show that 
$\widetilde{e}_Y:\mathcal{I}(\Sigma, \mathcal{P})\rightarrow \text{Hom}(H_1^{\mathcal{P}}(\Sigma; \mathbb{Z}), \mathbb{Z})$ is the unique homomorphism satisfying 
the naturality property. Let $F: \mathcal{I}(\Sigma, \mathcal{P})\rightarrow \text{Hom}(H_1^{\mathcal{P}}(\Sigma; \mathbb{Z}), \mathbb{Z})$ be another homomorphism such that 
$\widetilde{e}_{\widehat{Y}}\circ \widehat{j}_{\ast}= (\psi^*_{\mathscr{K}})^{-1}\circ F$.
Recall that $(\psi^*_{\mathscr{K}})^{-1}: \text{Hom}(H_1^{\mathcal{P}}(\Sigma; \mathbb{Z}), \mathbb{Z})\rightarrow \text{Hom}(H_1^{\widehat{\mathcal{P}}}(\widehat{\Sigma}; \mathbb{Z}), \mathbb{Z})$ is 
defined such that 
$(\psi^*_{\mathscr{K}})^{-1}(\chi)= \chi\circ\psi^{-1}_{\mathscr{K}}$ for any 
$\chi\in \text{Hom}(H_1^{\mathcal{P}}(\Sigma; \mathbb{Z}), \mathbb{Z})$. Observe that 
$(\psi^*_{\mathscr{K}})^{-1}$ is an isomorphism because 
$\psi_{\mathscr{K}}$ is an isomorphism. 
Hence by composing both sides of 
$(\psi^*_{\mathscr{K}})^{-1}\circ \widetilde{e}_Y=(\psi^*_{\mathscr{K}})^{-1}\circ F$ 
with $\psi^*_{\mathscr{K}}$, we get the equality $\widetilde{e}_Y= F$. 

This finishes the proof.
\end{proof}

\subsection{Naturality of the Chillingworth Homomorphism}
\label{section:naturality2}
In this section, we show that the Chillingworth homomorphism is natural. 
We relate the Chillingworth class of the subsurface Torelli 
group to the partitioned Johnson homomorphism.

For an element $f\in \mathcal{I}(\Sigma, \mathcal{P})$, let us define the dual of 
$\widetilde{e}_Y(f)$. This will be called the Chillingworth class of $f$. The algebraic intersection form for 
$H_1^{\mathcal{P}}(\Sigma; \mathbb{Z})$ gives $t_{(\Sigma, \mathcal{P})}(f)$ defined by: 
\[
\widehat{i}([\gamma], t_{(\Sigma, \mathcal{P})}(f))= \widetilde{e}_Y(f)[\gamma].
\]
Therefore, we get the Chillingworth homomorphism:
\[
t_{(\Sigma, \mathcal{P})}: \mathcal{I}(\Sigma, \mathcal{P})\rightarrow H_1^{\mathcal{P}}(\Sigma; \mathbb{Z}).
\]

Let $(\Sigma, \mathcal{P})\hookrightarrow (\Sigma', \mathcal{P}')$ be 
an embedding of partitioned surfaces. Fix a symplectic basis $\mathscr{B}$ of
$H_1^{\mathcal{P}}(\Sigma; \mathbb{Z})$ defined in Section~\ref{section:symplecticbasis}. 
Recall that $H_1^{\mathcal{P}'}(\Sigma'; \mathbb{Z})$ 
is isomorphic to $H_1^{\widehat{\mathcal{P}}}(\widehat{\Sigma}; \mathbb{Z})\oplus H_1^{\mathcal{P}_1}(V_1; \mathbb{Z}) \oplus H_1^{\mathcal{P}_2}(V_2; \mathbb{Z})\oplus\cdots\oplus H_1^{\mathcal{P}_n}(V_n; \mathbb{Z})$ as in Remark~\ref{rmk:homology}. 

As in the previous section, take a nonvanishing vector field $X$ on $\Sigma'$. Restrict $X$ to the subsurface $\Sigma$ and call the restriction $Y$.

\begin{lemma} 
\label{lem:commutativity}
Let $s_{\ast}: H_1^{\widehat{\mathcal{P}}}(\widehat{\Sigma}; \mathbb{Z})\rightarrow H_1^{\mathcal{P}'}(\Sigma'; \mathbb{Z})$ be the inclusion map and $D$ be the 
isomorphism defined in Section~\ref{section:symplecticbasis}. 
Then the following diagram commutes:
\begin{equation}
\xymatrix{
 \mathcal{I}(\Sigma, \mathcal{P}) \ar[r]^{i_\ast} \ar[d]_{\widetilde{e}_Y} &
                                        \mathcal{I}(\Sigma', \mathcal{P}')\ar[d]^{\widetilde{e}_X} \\
 \textnormal{Hom} (H_1^{\mathcal{P}}(\Sigma; \mathbb{Z}), \mathbb{Z}) \ar[r]^{i'_\ast} \ar[d]^{D^{-1}}
                  & \textnormal{Hom} (H_1^{\mathcal{P}'}(\Sigma'; \mathbb{Z}), \mathbb{Z}) \ar[d]^{D^{-1}} \\
 H_1^{\mathcal{P}}(\Sigma; \mathbb{Z}) \ar[r]^{s_{\ast}\circ\psi_{\mathscr{K}}}                      &H_1^{\mathcal{P}'}(\Sigma'; \mathbb{Z})} 
 \label{eq:4}
\end{equation}
\end{lemma}

\begin{proof}
We showed in Theorem \ref{thm:natural} that the upper square in 
diagram~(\ref{eq:4}) commutes. Hence our aim is to show that the lower 
square also commutes. Here, the image of 
$i'_{\ast}: \text{Hom}(H_1^{\mathcal{P}}(\Sigma; \mathbb{Z}), \mathbb{Z})\rightarrow \text{Hom}(H_1^{\mathcal{P}'}(\Sigma'; \mathbb{Z}), \mathbb{Z})$ 
is defined as the composition of $(\psi^*_{\mathscr{K}})^{-1}$ and $r^*$ as before. 
Let $\chi\in \text{Hom}(H_1^{\mathcal{P}}(\Sigma; \mathbb{Z}), \mathbb{Z})$. 
Then $i'_{\ast}(\chi)[\gamma]= \chi(\psi^{-1}_{\mathscr{K}} r_{\ast})[\gamma]$ for any 
$[\gamma]\in H_1^{\mathcal{P}'}(\Sigma'; \mathbb{Z})$. Recall that $r_{\ast}$ 
is the projection of $H_1^{\mathcal{P}'}(\Sigma'; \mathbb{Z})$ on $H_1^{\widehat{\mathcal{P}}}(\widehat{\Sigma}; \mathbb{Z})$. Commutativity of the lower square is 
proven by showing commutativity of diagram~(\ref{eq:5}).
\begin{equation}
 \xymatrix{
 \text{Hom} (H_1^{\mathcal{P}}(\Sigma; \mathbb{Z}), \mathbb{Z}) \ar[r]^{(\psi^*_{\mathscr{K}})^{-1}} & \text{Hom} (H_1^{\widehat{\mathcal{P}}}(\widehat{\Sigma}; \mathbb{Z}), \mathbb{Z}) \ar[r]^{r^{\ast}} 
                                            & \text{Hom} (H_1^{\mathcal{P}'}(\Sigma'; \mathbb{Z}),   \mathbb{Z})\\
 H_1^{\mathcal{P}}(\Sigma; \mathbb{Z}) \ar[r]^{\psi_{\mathscr{K}}} \ar[u]^{D}     &H_1^{\widehat{\mathcal{P}}}(\widehat{\Sigma}; \mathbb{Z}) \ar[r]^{s_{\ast}} \ar[u]^{D} & H_1^{\mathcal{P}'}(\Sigma'; \mathbb{Z}) \ar[u]^{D}} 
 \label{eq:5}
\end{equation} 

We will analyze $2$ cases for the square on the left of diagram (\ref{eq:5}). 

By definition, $\psi^{-1}_{\mathscr{K}}: H_1^{\widehat{\mathcal{P}}}(\widehat{\Sigma}; \mathbb{Z})\rightarrow H_1^{\mathcal{P}}(\Sigma; \mathbb{Z})$ 
preserves the algebraic intersection form, 
i.e. for any $a, b\in H_1^{\widehat{\mathcal{P}}}(\widehat{\Sigma}; \mathbb{Z})$ 
we have 
$\widehat{i}(a, b)= \widehat{i}(\psi^{-1}_{\mathscr{K}}(a), \psi^{-1}_{\mathscr{K}}(b))$.

\underline{Case $1$}: 
For any homology class $[x]$ of 
$H_1^{\mathcal{P}}(\Sigma; \mathbb{Z})$ with a closed curve representative, we have 
\begin{eqnarray}\label{eq:001}
(\psi^{\ast}_{\mathscr{K}})^{-1}(D([x]))[\gamma]= D([x])(\psi^{-1}_{\mathscr{K}}([\gamma]))= \widehat{i}(\psi^{-1}_{\mathscr{K}}[\gamma], [x])
\end{eqnarray}
and 
\begin{eqnarray}\label{eq:002}
D(\psi_{\mathscr{K}}([x]))[\gamma]= D([x])([\gamma])= \widehat{i}([\gamma], [x]).
\end{eqnarray}

For simplicity, $[x]$ denotes both an element in 
$H_1^{\mathcal{P}}(\Sigma; \mathbb{Z})$ and its image under the isomorphism 
$\psi_{\mathscr{K}}$. In (\ref{eq:001}), $[x]$ is an element of 
$H_1^{\mathcal{P}}(\Sigma; \mathbb{Z})$, whereas in (\ref{eq:002}), 
$[x]$ is an element of 
$H_1^{\widehat{\mathcal{P}}}(\widehat{\Sigma}; \mathbb{Z})$. 
Hence, we can consider $[x]$ in (\ref{eq:001}) to be 
$\psi^{-1}_{\mathscr{K}}([x])$. 
The following gives the commutativity for $[x]$:
\begin{eqnarray}
(\psi^{\ast}_{\mathscr{K}})^{-1}(D([x]))[\gamma]
& = & \widehat{i}(\psi^{-1}_{\mathscr{K}}[\gamma], [x]) \nonumber \\
& = & \widehat{i}(\psi^{-1}_{\mathscr{K}}[\gamma], \psi^{-1}_{\mathscr{K}}([x])) \nonumber \\
& = & \widehat{i}([\gamma], [x]) \nonumber \\
& = & D(\psi_{\mathscr{K}}([x]))[\gamma] \nonumber
\end{eqnarray}
for all $[\gamma]$ in $H_1^{\widehat{\mathcal{P}}}(\widehat{\Sigma}; \mathbb{Z})$.

\underline{Case $2$}: 
For the basis elements $[h_j^l]$ of $H_1^{\mathcal{P}}(\Sigma; \mathbb{Z})$, 
we will show that left-hand side of diagram (\ref{eq:5}) commutes. We have
\begin{eqnarray}\label{eq:003}
(\psi^{\ast}_{\mathscr{K}})^{-1}(D([h_j^l]))[\gamma]= D([h_j^l])(\psi^{-1}_{\mathscr{K}}([\gamma]))= \widehat{i}(\psi^{-1}_{\mathscr{K}}[\gamma], [h_j^l])
\end{eqnarray}
and 
\begin{eqnarray}\label{eq:004}
D(\psi_{\mathscr{K}}([h_j^l]))[\gamma]= D([h_j^l\ast k_j^l])([\gamma])= \widehat{i}([\gamma], [h_j^l\ast k_j^l]).
\end{eqnarray}
As in Case $1$, the last term in (\ref{eq:003}) denotes the algebraic 
intersection number in $H_1^{\mathcal{P}}(\Sigma; \mathbb{Z})$ and the last term 
in (\ref{eq:004}) denotes the algebraic intersection number in 
$H_1^{\widehat{\mathcal{P}}}(\widehat{\Sigma}; \mathbb{Z})$. 

By the same reasoning as in the previous case, 
\[
\widehat{i}(\psi^{-1}_{\mathscr{K}}[\gamma], [h_j^l])= \widehat{i}(\psi^{-1}_{\mathscr{K}}[\gamma], \psi^{-1}_{\mathscr{K}}([h_j^l\ast k_j^l]))=  \widehat{i}([\gamma], [h_j^l\ast k_j^l])
\]
for all $[\gamma]$ in $H_1^{\widehat{\mathcal{P}}}(\widehat{\Sigma}; \mathbb{Z})$.

Finally, for the left-hand side of diagram (\ref{eq:5}) we have 
$(\psi^{\ast}_{\mathscr{K}})^{-1}\circ D= \psi_{\mathscr{K}}\circ D$. 
This shows commutativity for the 
left-hand side of diagram (\ref{eq:5}).

Now our aim is to show that the square in the right-hand side of 
diagram (\ref{eq:5}) commutes. Recall that $s_{\ast}: H_1^{\widehat{\mathcal{P}}}(\widehat{\Sigma}; \mathbb{Z})\rightarrow H_1^{\mathcal{P}'}(\Sigma'; \mathbb{Z})$ is the inclusion map.

Let $[x]$ be an element of $H_1^{\widehat{\mathcal{P}}}(\widehat{\Sigma}; \mathbb{Z})$. 
For any homology basis element $[\gamma]\in H_1^{\mathcal{P}'}(\Sigma'; \mathbb{Z})$ 
the lemma follows:
\[
r^{\ast}(D([x]))[\gamma]= D([x])(r_{\ast}([\gamma]))=\widehat{i}(r_{\ast}([\gamma]), [x])
\]
\[
D(s_{\ast}([x]))[\gamma]= D([x])([\gamma])= \widehat{i}([\gamma], [x]).
\]  
If a representative of $[\gamma]$ is contained in the complement of 
$\widehat{\Sigma}$, $r_{\ast}([\gamma])=0$. Hence 
$r^{\ast}(D([x]))[\gamma]=0$. 
Since a representative of $[x]$ is contained in $\widehat{\Sigma}$, 
$D(s_{\ast}([x]))[\gamma]=0$ 
is obtained. 

If a representative of $[\gamma]$ is contained in $\widehat{\Sigma}$, 
$r_{\ast}([\gamma])=[\gamma]$ and so 
$r^{\ast}(D([x]))[\gamma]= D(s_{\ast}([x]))[\gamma]$ 
as desired.

Consequently, we have proven that diagram (\ref{eq:5}) commutes. Since the dual 
maps $D$ are isomorphisms, we obtain that diagram (\ref{eq:4}) is also commutative.
We conclude that $s_{\ast}\circ\psi_{\mathscr{K}}\circ t_{(\Sigma, \mathcal{P})}=t_{(\Sigma', \mathcal{P}')}\circ i_{\ast}$ 
by diagram (\ref{eq:4}).
\end{proof}

\begin{corollary}
\label{cor:bigdiagram} 
The following diagram is commutative and hence we get the following equality: 
$t_{(\Sigma,\mathcal{P})}=\psi^{-1}_{\mathscr{K}}\circ r_{\ast}\circ C\circ p_i\circ \tau_{(\Sigma, \mathcal{P})}$, 
where $C$ is the contraction map. Here, $\tau_{(\Sigma, \mathcal{P})}$ is 
the partitioned Johnson homomorphism defined in \cite{Church}, Definition $5.2$, and
$p_i$ is the map defined in \cite{Church}, Definition $5.13$.
\begin{equation}
\xymatrix{
  W_{(\Sigma, \mathcal{P})}\ar[drr]^{p_i} & &\\
 \mathcal{I}(\Sigma, \mathcal{P}) \ar@/_5pc/@{.>}[dd]_{t_{(\Sigma,\mathcal{P})}} \ar[u]^{\tau_{(\Sigma, \mathcal{P})}} \ar[r]^{i_\ast} \ar[d]_{\widetilde{e}_Y} &
                                          \mathcal{I}(\Sigma_{g,1})\ar[d]^{\widetilde{e}_X} \ar@{.>}[dr]^t \ar[r]^{\tau} & \bigwedge^3 H_1(\Sigma_{g,1}; \mathbb{Z})\ar[d]^C \\
\textnormal{Hom} (H_1^{\mathcal{P}}(\Sigma; \mathbb{Z}), \mathbb{Z}) \ar[r]^{i'_\ast} \ar[d]_{D^{-1}}                         
                          & H^1(\Sigma_{g,1}; \mathbb{Z}) \ar[r]^{D^{-1}} &    H_1(\Sigma_{g,1}; \mathbb{Z}) \ar[dll]^{\psi^{-1}_{\mathscr{K}}\circ r_{\ast}} \\
 H_1^{\mathcal{P}}(\Sigma; \mathbb{Z}) } 
 \label{eq:6}
\end{equation}

\end{corollary}

\begin{proof} 
We need to confirm that each triangle and square is commutative. 
The partitioned Johnson homomorphism is natural \cite{Church}, Theorem $5.14$. 
Hence, the upper triangle is commutative. We showed in Theorem~\ref{thm:natural} 
that the left square in the middle part is commutative. The commutativity of the 
right square in the middle follows from Theorem $2$ of \cite{Johnson} and the definition 
of the Chillingworth class. Finally, the commutativity 
of the lower triangle follows from Lemma \ref{lem:commutativity}.
\end{proof}

\providecommand{\bysame}{\leavevmode\hbox
to3em{\hrulefill}\thinspace}
\providecommand{\MR}{\relax\ifhmode\unskip\space\fi MR }
 \MRhref  \MR
\providecommand{\MRhref}[2]{
  \href{http://www.ams.org/mathscinet-getitem?mr=#1}{#2}
 } \providecommand{\href}[2]{#2}

\bibliographystyle{abbrv}	

\bibliography{paperhatice}

\end{document}